\documentclass[reqno]{amsart}
\usepackage{mathrsfs}
\usepackage{amsmath}
\usepackage{amsthm}
\usepackage{amsfonts}
\usepackage{amssymb}
\usepackage{url}
\usepackage{enumerate}
\usepackage[pdftex,bookmarks=true]{hyperref}

\newcommand\R{{\mathbf{R}}}

\renewcommand\P{{\mathbf{P}}}
\newcommand\E{{\mathbf{E}}}

\newcommand\dist{{\operatorname{dist}}}

\newcommand\col{{\mathbf{c}}}
\newcommand\row{{\mathbf{r}}}

\newcommand\ep{\varepsilon}

\newcommand\Bc{{\mathbf c}}

\newcommand\Bu{{\mathbf u}}
\newcommand\Bv{{\mathbf v}}

\newcommand\Bx{{\mathbf x}}
\newcommand\By{{\mathbf y}}
\newcommand\Bz{{\mathbf z}}

\newcommand\CE{{\mathcal E}}
\newcommand\CF{{\mathcal F}}
\newcommand\CG{{\mathcal G}}
\newcommand\CH{{\mathcal H}}

\newcommand\Comp{\mathbf{Comp}}
\newcommand\Incomp{\mathbf{Incomp}}
\newcommand\supp{\mathbf{supp}}

\newcommand\eps{\varepsilon}

\theoremstyle{plain}
  \newtheorem{theorem}[subsection]{Theorem}

  \newtheorem{fact}[subsection]{Fact}
  \newtheorem{lemma}[subsection]{Lemma}
  \newtheorem{corollary}[subsection]{Corollary}

  \newtheorem{remark}[subsection]{Remark}
  
  \newtheorem{claim}[subsection]{Claim}

\theoremstyle{definition}
  \newtheorem{definition}[subsection]{Definition}

\begin{document}
\title{Random matrices: overcrowding estimates for the spectrum}

\author{Hoi H. Nguyen}
\email{nguyen.1261@math.osu.edu}
\address{Department of Mathematics, The Ohio State University, 231 West 18th Avenue, Columbus, OH 43210}

\thanks{The author is supported by research grant DMS-1600782}




\begin{abstract} We address overcrowding estimates for the singular values of random iid matrices, as well as for the eigenvalues of random Wigner matrices. We show evidence of long range separation under arbitrary perturbation even in matrices of discrete entry distributions. In many cases our method yields nearly optimal bounds.

\end{abstract}

\maketitle

\section{introduction}

\subsection{Random iid matrices with subgaussian tails}

Consider a random matrix $M=(m_{ij})_{1\le i,j\le n}$, where $m_{ij}$ are iid copies of a random variable $\xi$ of mean zero and variance one. Let $\sigma_n \le \dots \le \sigma_1$ be the singular values of $M$. 
 
An important problem with practical applications is to bound the condition number
of $M$.  As the asymptotic behavior of the largest singular value $\sigma_1$ is well understood under natural assumption on $\xi$,  the main problem is to study the lower bound of the least singular value $\sigma_n$.  This problem was  first raised by
Goldstine and von Neumann \cite{GN} well back in the 1940s, with connection to their investigation
of the complexity of inverting a matrix.

 To answer Goldstine and von Neumman's question, Edelman \cite{Edelman}  computed the distribution of the least singular value of Ginibre matrix (where $\xi$ is standard gaussian). He showed that for all fixed  $\ep >0$
$$ \P( \sigma_n  \leq \ep n^{-1/2}  ) = \int_0^{\ep^2}  \frac{1+\sqrt{x}}{2\sqrt{x}} e^{-(x/2 + \sqrt{x})}\ dx + o(1) =  \ep - \frac{1}{3} \ep^{3 }  +O(\ep^4) +o(1) . $$
Note that the same asymptotic continues to hold for any $ \ep>0$ which can go to zero with $n$ (see also \cite{ST})
\begin{equation}\label{eqn:Edelman}
\P (\sigma_n   \le \ep n^{-1/2} )  \le \ep.
\end{equation} 
For other singular values of Ginibre matrices, an elegant result by Szarek \cite{Szarek} shows that the $\sigma_{n-k+1}$ are separated away from zero with an extremely fast rate.

\begin{theorem}\label{theorem:Ginibre} Assume that $\xi$ is standard gaussian, then there exist absolute constants $C_1, C_2$ such that for all $\eps>0$, and all $1\le k\le n$
$$ (\frac{C_1}{k} \eps)^{k^2} \le \P( \sigma_{n-k+1} \le \frac{\eps}{\sqrt{n}}) \le (\frac{C_2}{k} \eps)^{k^2}.$$
\end{theorem}

In what follows $\eps$ is always bounded by $O(1)$. Motivated by the universality phenomenon in random matrix theory, we expect similar repulsion bounds for general random matrix ensembles. More specifically, we will assume $\xi$ to have {\it mean zero, variance one}, and subgaussian distribution. In other words, there exists $B>0$ such that 
\begin{equation}\label{eqn:xi}
\P(|\xi|>t) \le 2 \exp(-t^2/B^2) \mbox{ for all } t.
\end{equation}
The smallest of such $B$ is called the subgaussian moment of $\xi$. A representative example of our study, on the opposite side of being gaussian, is the Bernoulli (Radamacher) random variable which takes value $\pm 1$ with probability 1/2.

When $k=1$, Rudelson and Vershynin \cite{RV} extended \eqref{eqn:Edelman} to more general iid random matrices.
\begin{theorem}\label{theorem:RV}
 Let $M=(m_{ij})_{1\le i,j\le n}$ be a random matrix where $m_{ij}$ are iid copies of a subgaussian random variable $\xi$  as in \eqref{eqn:xi}. Then for any $\eps>0$ that might depend on $n$ we have 
$$\P(\sigma_{n} \le \frac{\eps}{\sqrt{n}}) \le C  \eps + \exp(-cn),$$
where $C$ and $c$ depend only on the subgaussian moment of $\xi$.
\end{theorem}

Furthermore, it was shown by Tao and Vu \cite{TVleast} that the statistics of $\sqrt{n}\sigma_n$ is universal. Thus our understanding in the case $k=1$ is nearly complete.

In this note we address the overcrowding direction by investigating Theorem \ref{theorem:Ginibre} under various settings.

\begin{theorem}[the iid case]\label{theorem:main} Let $M=(m_{ij})_{1\le i,j\le n}$ be a random matrix where $m_{ij}$ are iid copies of a subgaussian random variable $\xi$ as in \eqref{eqn:xi}. For any $k\ge 1$ there exist a constant $C_k$ depending on $k$ and a constant $c$ depending only on the subgaussian moment of $\xi$ such that for any $\eps>0$,
 \begin{equation}\label{eqn:k^2}
 \P(\sigma_{n-k+1} \le \frac{\eps}{\sqrt{n}}) \le C_k \eps^{k^2} + \exp(-cn).
\end{equation} 
Furthermore, for any $0<\gamma<1$, there exist $C,c$ and $\gamma_0$ such that for $\gamma_0^{-1} < k< \gamma_0 n$
 \begin{equation}\label{eqn:assym} 
 \P(\sigma_{n-k+1} \le \frac{\eps}{\sqrt{n}}) \le (\frac{C\eps}{k})^{(1-\gamma)k^2} + \exp(-cn).
 \end{equation}
Equivalently, with $k < \gamma_0 n$, for any $0<\eps<1$ let $I$ be the interval $[0, k \eps/\sqrt{n}]$ and $N_I$ be the number of singular values belonging to $I$. Then we obtain the following (overcrowding) Wegner-type estimate at the hard edge
\begin{equation}\label{eqn:overcrowding}
\P(N_I \ge k)  \le (C\eps)^{(1-\delta)k^2} + \exp(-cn).
\end{equation}
\end{theorem}

Estimate \eqref{eqn:overcrowding} improves \cite[Proposition 4.1]{CMS} of Cacciapuoti, Maltsev and Schlein where they showed $\P(N_I \ge k)  = O(\eps^{Ck})$ with the assumption that $\xi$ has bounded density and subgaussian tail. Under this assumption we can omit the additional terms $\exp(-cn)$ above, see Remark \ref{remark:iid:cont}.

\subsection{Perturbation of random iid matrices} In connection to Edelman's formula for Ginibre ensemble, and motivated by the study of smoothed analysis, Sankar, Spielman and Teng \cite{ST, SST} have found the following striking phenomenon. 

\begin{theorem}\label{theorem:SST} Let $M$ be a Ginibre ensemble, and let $F$ be any deterministic matrix. Then for any $\eps > 0$
$$\P (\sigma_n(M+F) \le  \frac{\eps}{\sqrt{n}} ) =O(\eps).$$
\end{theorem}
Thus, no matter deterministic matrix $F$ we perturb, the gaussian randomness regularizes it out in such a way that $\sigma_n(M+F)$ behaves similar to $\sigma_n(M)$. Interestingly, this effect no longer holds if the entries of $M$ are allowed to have discrete distributions, see for instance the construction by Tao and Vu in \cite{TVsmooth}, or by Rudelson in \cite{FV}. In fact in the latter example, one can have iid $M$ and deterministic $\|F\|_2=N$ for any $N$ such that $\P(\sigma_n(M+F) \le \sqrt{n}/N) \ge 1/2$.

Although one can still get useful bounds on $\sigma_n(M+F)$ when $\|M\|_2=n^{O(1)}$ (see for instance \cite{TVsmooth} by Tao and Vu), these examples just demonstrate that there is no universal $F$-independent asymptotic behavior of $\sigma_n(X+F)$ in terms of randomness. 

However, we might still ask: \\

\centerline{\it What about the local interaction of the eigenvalues (singular
  values)?}

Our results below, Theorem \ref{theorem:main:perturb:iid} and Corollary \ref{cor:main:perturb:iid}, support the phenomenon that a major part of the eigenvalues (singular values) stays well-separated under arbitrary perturbation even in random matrices having discrete entry distributions.

\begin{theorem}[the perturbed case for iid matrices]\label{theorem:main:perturb:iid} Let $M$ be a random matrix where the entries are iid copies of a random variable $\xi$ of variance one. Let $F$ be any deterministic matrix. There exists an absolute constant $C$ such that for any $\eps>0$ and for any $k\ge 1$ we have
 \begin{equation}\label{eqn:k^2:perturb:iid}
 \P\left( \sigma_{n-k+1}(M+F) \le \frac{\eps}{\sqrt{n}} \right) \le n^{k-1}  (C kp)^{(k-1)^2}
\end{equation} 
where  $p=p(\eps)=\sup_{x\in \R}\P(|\xi-x| \le \eps)$.

Furthermore, for given $0<\gamma<1$, there exists a constant  $C=C(\gamma)$ such that for any $k$ we have
 \begin{equation}\label{eqn:assym:perturb:iid} 
  \P\left( \sigma_{n-k+1}(M+F) \le \frac{\eps k}{\sqrt{n}} \right)  \le n^{\lfloor (1-\gamma/2)k \rfloor} (Cp)^{(1-\gamma)k^2}.
 \end{equation}

 Equivalently, with $I$ being the interval $[0, \frac{\eps k}{\sqrt{n}}]$ and $N_I$ be the number of singular values of $M+F$ in $I$ 
\begin{equation}\label{eqn:overcrowding:perturb:iid}
\P(N_I \ge k)  \le n^{\lfloor (1-\gamma/2)k \rfloor} (Cp)^{(1-\gamma)k^2}.
\end{equation}

Finally, if $|\xi|=O(1)$ with probability one then there exist constants $K$ and $c_1,c_2$ depending on $\xi$ such that for any $k>K$ we have
 \begin{equation}\label{eqn:discrete:perturb:iid} 
  \P\left( \sigma_{n-k+1}(M+F) \le \frac{c_1 k}{\sqrt{n}} \right)  \le n^{k} e^{-c_2 k^2}.
 \end{equation}

\end{theorem}
It is not clear what would be the optimal exponents on the multiplicative factors $n^{O(.)}$ above (and also in Theorem \ref{theorem:main:perturb:sym} below). With our current method, the more (delocalization) information we know about the singular vectors (the eigenvectors), the better exponents we could obtain. We next deduce two consequences.
\begin{corollary}\label{cor:main:perturb:iid} Let $M$ be a random matrix where the entries are iid copies of a random variable $\xi$ of variance one. Let $F$ be any deterministic matrix. 

\begin{itemize}
\item If the common distribution $\xi$ has a density function bounded by $K$, then for given $0<\gamma<1$ there exists a constant  $C=C(\gamma)$ such that for any $\eps>0$ and $k\ge 1$, with $I=[0, \frac{\eps k}{\sqrt{n}}]$
\begin{equation}\label{eqn:overcrowding:perturb:cont:iid}
\P(N_I \ge k)  \le n^k (CK \eps )^{(1-\gamma)k^2}.
\end{equation} 
\item If $\xi$ has discrete distribution, then there exist constants $C,c_1,c_2$ depending on $\xi$ such that for any $k\ge C \log n$, with $I=[0, \frac{c_1k }{\sqrt{n}}]$
\begin{equation}\label{eqn:overcrowding:perturb:discrete:iid}
\P(N_I \ge k)  \le e^{-c_2k^2}.
\end{equation} 
\end{itemize}
\end{corollary}

The quadratic rate in \eqref{eqn:overcrowding:perturb:discrete:iid} is consistent with Theorem \ref{theorem:Ginibre}. Furthermore, the bound is valid for any non-degenerate discrete distribution and any $F$. 



\subsection{Symmetric Wigner matrices with sub-gaussian tails} A symmetric Wigner matrix $X$ of size $n$ with sub-gaussian tail is a random symmetric matrix whose strictly upper triangular entries are iid copies of a real-valued random variable $\xi$ as in \eqref{eqn:xi}, and whose diagonal entries are independent sub-gaussian random variables with mean zero and  variances bounded by $n^{1-o(1)}$, with the diagonal entries independent of the strictly upper triangular entries.  

Similarly to Theorem \ref{theorem:RV}, the following result was shown by Vershynin \cite{Vershynin}.
\begin{theorem}\label{theorem:V}  With $X=(x_{ij})_{1\le i,j\le n}$ a Wigner matrix of subgaussian entries  as in \eqref{eqn:xi}, there exists an absolute constant $c<1$ such that for any fixed real number $z$ and any $\eps>0$ we have
$$\P(\min_k |\lambda_k(X) -z| \le \frac{\eps}{\sqrt{n}}) =O(\eps^{1/9} + e^{-n^c}).$$
\end{theorem}

This bound does not seem to be optimal, in fact the RHS is conjectured to be similar to that  of Theorem \ref{theorem:RV} \cite{Vershynin}. Note that it follows from a result by  Bourgade, Erd\H{o}s, Yau and Yin \cite{BEYY} that the distribution of $\min_k \sqrt{n}|\lambda_k(X)|$ is universal. 

Under some strong smoothness and decay hypotheses on the entries of a symmetric Wigner matrix $X$, but without the assumption that the entries have the same variance, a near optimal Wegner-type estimate has been shown in \cite[Theorem B1]{BEYY} (see also \cite{ErdSchYau2010}).
\begin{theorem}\label{theorem:BEYY} Let $X=(x_{ij})_{1\le i,j\le n}$ be a symmetric Wigner matrix with entries of finite $p$-moment for some sufficiently large $p$, and $G$ be a GOE matrix. For any $t>0$ we denote $\lambda_1(t)\le \dots \le \lambda_n(t)$ the eigenvalues of $\sqrt{1-t} X+ \sqrt{t} G$. Define the set 
$$\CG_\delta:=\Big\{|\lambda_i - \gamma_i| \le n^{-2/3+\delta} (\min(i,n+1-i))^{-1/3} \mbox{ for all } i \in [n]\Big\},$$
where $\gamma_i$ denotes the classical location of the $i$-th eigenvalue under the semicircle law ordered in increasing order. For any fixed $\kappa$ there exists $C_1>0$ such that for any $k\ge 1, \tau,\delta>0$ there exists $C_2>0$ such that the following holds for any $z \in (-2+\kappa, 2-\kappa)/n^{1/2}, t\in [n^{-\tau},1]$ and $\eps>0$ 
\begin{equation}\label{eqn:esy-est}
 \P\left(  \CG_\delta \wedge z- \frac{\eps}{\sqrt{n}} \leq \lambda_i \leq \lambda_{i+k-1} \leq z +  \frac{\eps}{\sqrt{n}} \hbox{ for some } i  \right) \le C_2 n^{2 k \delta + C_1 k^2 \tau} \eps^{k(k+1)/2}. 
\end{equation}
\end{theorem}
 
 
Note that aside from the correcting factor $n^{2 k \delta + C_1 k^2 \tau}$, the bound $\eps^{k(k+1)/2}$ in \eqref{eqn:esy-est} above is optimal for any fixed $k$. 

Here we show a variant of Theorem \ref{theorem:BEYY} only under the subgaussian decay of the entries. We will also address the case that $k$ might vary together with $n$, which seems to be new even for smooth ensembles. 
 \begin{theorem}[the symmetric case]\label{theorem:main:sym} Let $X=(x_{ij})_{1\le i,j\le n}$ be a Wigner matrix of subgaussian entries as in \eqref{eqn:xi}. For any $k\ge 1$ there exist a constant $C_k$ depending on $k$ and a constant $c$ depending only on the subgaussian moment of $\xi$ such that for any $\eps>0$,
 \begin{equation}\label{eqn:k^2:sym}
 \P\left(z - \frac{\eps}{\sqrt{n}} \leq \lambda_i \leq \lambda_{i+k-1} \leq z  + \frac{\eps}{\sqrt{n}} \hbox{ for some } i \right) \le C_k \eps^{k(k-1)/2}+  \exp(-n^c). 
\end{equation} 
Furthermore, for given $0<\gamma<1$, there exist $C,c$ and $\gamma_0$ such that for $\gamma_0^{-1}<k< \gamma_0 n$
 \begin{equation}\label{eqn:assym:sym} 
  \P\left( z - \frac{\eps}{\sqrt{n}} \leq \lambda_i \leq \lambda_{i+k-1} \leq z + \frac{\eps}{\sqrt{n}} \hbox{ for some } i \right)  \le (\frac{C\eps}{k})^{(1-\gamma)k^2/2} +  \exp(-n^c).
 \end{equation}
Equivalently, let $I$ be the interval $[z- \frac{\eps k}{\sqrt{n}}, z+ \frac{\eps k}{\sqrt{n}}]$ and $N_I$ be the number of eigenvalues belonging to $I$. Then we have the following Wegner-type estimate 
\begin{equation}\label{eqn:overcrowding:sym}
\P(N_I \ge k)  \le (C\eps)^{(1-\gamma)k^2/2} + \exp(-n^c).
\end{equation}
\end{theorem}

If we work with Hermitian Wigner matrices instead (where the real and imaginary parts of the off diagonal entries are iid subgaussian of zero mean and variance 1/2) then the exponents of the bounds above are doubled. Furthermore, the additive term $\exp(-n^c)$ can be omitted if the subgaussian random variable has bounded density function (Remark \ref{remark:sym:cont}). 
 
Notice also that unlike in the iid case, our repulsion result  is valid over any interval. By taking union bound over all $z=z_i=i \eps/\sqrt{n}, |i|=O(n \eps^{-1})$, we obtain the following bound on all gaps of given range (say for small $k$).
\begin{corollary} With the same assumption as in Theorem \ref{theorem:main:sym},
\begin{equation}\label{eqn:k^2:sym'}
 \P\Big(\hbox{There exists  $1\le i\le n-k+1$ such that }  \lambda_{i+k-1} -\lambda_i  \leq \frac{\eps}{\sqrt{n}} \Big) =O_k(n \eps^{k(k-1)/2-1})+  \exp(-n^c).
\end{equation}
\end{corollary}
This bound is vacuous for $k=1,2$ but becomes nearly optimal for large $k$. It is also comparable to \cite[Theorem 2.4]{NgTV} where the gaps $\lambda_{i+k-1}-\lambda_i$ were considered for given $i$.

\subsection{Perturbation of symmetric Wigner matrices}
For symmetric Wigner matrices, Theorem \ref{theorem:SST} was extended to GOE (and GUE) recently by Aizenman, Peled, Schenker, Shamis and S. Sodin in \cite{A} (see also a similar bound by Bourgain in \cite{Bourgain}). Furthermore, Farell and Vershynin \cite{FV} showed an  $F$-{\it independent} bound for $\min_i |\lambda_i(X+F)|$ whenever the upper diagonal entries of $X$ have bounded density. 

Here again, it is possible to modify the constructions for iid matrices to show that the phenomenon no longer holds when the entries of $M$ have discrete distributions (see for instance \cite[p. 19]{A}). Our question here, similarly to the iid case  is whether eigenvalue repulsion sustains perturbations.

To proceed further, we introduce a Wegner-type estimate by Aizenman, Peled, Schenker, Shamis and Sodin in \cite[(1.9)]{A}. 

\begin{theorem}\label{theorem:A} Let $G$ be a GOE matrix, and $F$ be any deterministic symmetric matrix of size $n$. Then for any $\eps>0$, for any interval $I \subset \R$ of length $\eps/\sqrt{n}$, the number of eigenvalues of $G+F$ in $I$ satisfies the following $F$-independent bound for any $k\ge 1$
$$\P(N_I \ge k) \le (\frac{C \eps}{k})^k,$$
where $C$ is an absolute constant.
\end{theorem}

It is remarked that although Aizenman et.~al.~considered Theorem \ref{theorem:A}, their primary focus was on the simultaneous distribution of eigenvalues from possibly different intervals. However, it seems that their method was designed only for GOE and GUE. Also, the repulsion rate should be quadratic in $k$ (as also remarked in \cite[p. 18]{A}).  Here we show the following.

\begin{theorem}[the perturbed case for symmetric matrices]\label{theorem:main:perturb:sym}Let $X$ be a random symmetric Wigner matrix  where the upper diagonal entries are iid copies of a random variable $\xi$ of variance one. Let $F$ be a deterministic symmetric matrix. There exists an absolute constant $C$ such that for any $\eps>0$, for any $z\in \R$, and for any $k\ge 1$ we have
 \begin{equation}\label{eqn:k^2:perturb:sym}
 \P\left(z - \frac{\eps}{\sqrt{n}} \leq \lambda_i \leq \lambda_{i+k-1} \leq z  + \frac{\eps}{\sqrt{n}} \hbox{ for some } i \right) \le n^{k-1}  (C kp)^{k(k-1)/2}
\end{equation} 
where  $p=p(\eps)=\sup_{x\in \R}\P(|\xi-x| \le \eps).$

Furthermore, for given $0<\gamma<1$, there exists a constant  $C=C(\gamma)$ such that for any $k$ we have
 \begin{equation}\label{eqn:assym:perturb:sym} 
  \P\left( z - \frac{\eps k}{\sqrt{n}} \leq \lambda_i \leq \lambda_{i+k-1} \leq z + \frac{\eps k}{\sqrt{n}} \hbox{ for some } i \right)  \le n^{\lfloor (1-\gamma/2)k \rfloor} (Cp)^{(1-\gamma)k^2/2}.
 \end{equation}
 Equivalently, with $I$ being the interval $[z- \frac{\eps k}{\sqrt{n}}, z+ \frac{\eps k}{\sqrt{n}}]$ 
\begin{equation}\label{eqn:overcrowding:perturb}
\P(N_I \ge k)  \le n^{\lfloor (1-\gamma/2)k \rfloor} (Cp)^{(1-\gamma)k^2/2}.
\end{equation}
Finally, if $|\xi|=O(1)$ with probability one then there exist constants $K$ and $c_1,c_2$ depending on $\xi$ such that for any $k>K$ we have
 \begin{equation}\label{eqn:discrete:perturb:symmetric} 
\P\left( z - \frac{c_1 k}{\sqrt{n}} \leq \lambda_i \leq \lambda_{i+k-1} \leq z + \frac{c_1 k}{\sqrt{n}} \hbox{ for some } i \right)   \le n^{k} e^{-c_2 k^2}.
 \end{equation} 
\end{theorem}
Our result automatically extends to the case of Hermitian Wigner matrices where the real and imaginary parts of the off diagonal entries are iid subgaussian of zero mean and non-zero variance. (In fact it seems plaussible to improve the bounds with a double exponent for this model, but we will not pursue this here.) Similarly to Corollary \ref{cor:main:perturb:iid}, we present two corollaries.
\begin{corollary}\label{cor:main:perturb:sym} Let $X$ be a random symmetric Wigner matrix  where the upper diagonal entries are iid copies of a random variable $\xi$ of variance one. Let $F$ be a deterministic symmetric matrix.

\begin{itemize}
\item If the common distribution $\xi$ has a density function bounded by $K$, then for given $0<\gamma<1$, there exists a constant  $C=C(\gamma)$ such that for any $\eps>0$ and $k\ge 1$, with $I$ being the interval $[z- \frac{\eps k}{\sqrt{n}}, z+ \frac{\eps k}{\sqrt{n}}]$
\begin{equation}\label{eqn:overcrowding:perturb:cont}
\P(N_I \ge k)  \le n^k (CK \eps)^{(1-\gamma)k^2/2}.
\end{equation} 
\item If $\xi$ has discrete distribution, then there exists a constant $C,c_1,c_2$ depending on $\xi$ such that for any $k\ge C \log n$, with $I$ being the interval $[z- \frac{c_1k }{\sqrt{n}}, z+ \frac{c_1k }{\sqrt{n}}]$
\begin{equation}\label{eqn:discrete:perturb:sym}
\P(N_I \ge k)  \le e^{-c_2k^2}.
\end{equation} 
\end{itemize}
\end{corollary}

\subsection{Proof method and presentation} Our method is simple. Roughly speaking it translates the estimates under consideration to the events of having multiple independent columns of small distances to another independent subspace. One then uses the given distance estimates to show that these events are unlikely. As such, the starting point in each proof will be quite similar, however the later steps will evolve differently depending on the models.

The rest of the note is organized as follows. We will introduce the necessary ingredients in Section \ref{section:lemmas}. The proof of Theorem \ref{theorem:main} is broken down into two parts, Subsection \ref{section:k^2} and Subsection \ref{section:l^2} are devoted to proving \eqref{eqn:k^2} and \eqref{eqn:assym} respectively. The proof of Theorem \ref{theorem:main:sym} is carried out in Subsections \ref{section:k^2:sym} and \ref{section:l^2:sym} in a similar fashion. The proof of Theorem \ref{theorem:main:perturb:iid} and Theorem \ref{theorem:main:perturb:sym} will be completed in Section \ref{section:main:perturb:iid} and Section \ref{section:main:perturb:sym} respectively. Finally, in the appendix sections we recast the proof of Theorem \ref{theorem:Ginibre}, of Theorem \ref{theorem:BEYY} and of Theorem \ref{theorem:A} by following the given references.

\subsection{Notation}\label{notation-sec}

We write $X = O(Y)$, $X \ll Y$, or $Y \gg X$ to denote the claim that $|X| \leq CY$ for some fixed $C$; this fixed quantity $C$ is allowed to depend on other fixed quantities such as the sub-gaussian parameter of $\xi$. We also write $X= \Theta(Y)$ to denote the claim that $X\ll Y$ and $Y \ll X$.

For a square matrix $A$ and a number $\lambda$, for short we will write $A-\lambda$ instead of $A-\lambda I_n$. All the norms in this note, if not specified, will be the usual $\ell_2$-norm. We use the notation $\row_i(A)$ and $\col_j(A)$ to denote its $i$-th row and $j$-th column respectively. 

For notational convenience, identical constants will be reused in various contexts; these are usually different if not specified otherwise.

\section{Some supporting lemmas}\label{section:lemmas}

First of all, for the unperturbed models considered in Theorem \ref{theorem:main} and Theorem \ref{theorem:main:sym} we will condition on the event $\CE_{bound}$ that 
\begin{equation}\label{eqn:norm}
\|M\|_2 = O(\sqrt{n}) \mbox{ and } \|X\|_2 = O(\sqrt{n}) .
\end{equation}
These hold with probability $1-\exp(-\Theta(n))$. We now introduce the definition of compressible and incompressible vectors from \cite{RV} by Rudelson and Vershynin.

\begin{definition}\label{def:comp} Let $c_0, c_1 \in (0,1)$ be two numbers (chosen depending on the sub-gaussian moment of $\xi$.) A vector $x \in \R^n$ is called {\it sparse} if $|\supp(x)| \le c_0 n$.
A vector $x \in S^{n-1}$ is called {\it compressible} if $x$ is within Euclidean distance $c_1$ from the set of all sparse vectors.
  A vector $x \in S^{n-1}$ is called {\it incompressible} if it is not compressible.
  
  The sets of compressible and incompressible vectors in $S^{n-1}$
  will be denoted by $\Comp(c_0, c_1)$ and $\Incomp(c_0, c_1)$ respectively.
\end{definition}
Regarding the behavior of $M \Bx$ and $(X-z)\Bx$ for compressible vectors, the following was proved in \cite{RV} and in \cite{Vershynin} (see also \cite{NgTV}) for iid matrices and symmetric Wigner matrices respectively. 
\begin{lemma}\label{lemma:comp}
There exist positive constants $c_0,c_1, c,\alpha$ depending on the subgaussian moment of $\xi$ such that the following holds 
\begin{itemize}
\item (iid matrices) For any fixed $\Bu\in \R^n$,
$$\P(\inf_{\Bx \in \Comp(c_0,c_1)} \|M \Bx-\Bu\|_2 \le c \sqrt{n}) \le \exp(-\alpha n );$$
\item (symmetric matrices) For any fixed $z$,
$$\P(\inf_{\Bx \in \Comp(c_0,c_1)} \|(X-z) \Bx\|_2 \le c \sqrt{n}) \le \exp(-\alpha n ).$$
\end{itemize}
\end{lemma}

Our next ingredients are results about generic behavior of subspaces spanned by the matrix columns. We first cite a result by Rudelson and Vershynin \cite[Theorem 4.1]{RV-rec} for the iid model.

\begin{theorem}\label{theorem:dist} Let $H$ be a random subspace in $\R^n$ spanned by any $n-k$ column vectors of $M$ from Theorem \ref{theorem:main}, with $1\le k< \gamma_0 n$ for some constant $\gamma_0$ depending on the subgaussian moment of $\xi$.  There exists a deterministic subset $\CH$ of the Grassmanian of $\R^n$ where 
$$\P(H^\perp \in \CH) \ge 1- \exp(-c'n)$$
and such that for any subspace $H$ with $H^\perp \in \CH$ we have 
 \begin{equation}\label{eqn:dist}
\P_{\Bc}(\dist(\Bc, H) \le \delta) \le (\frac{C_0 \delta}{\sqrt{k}})^k + \exp(-c_0n),
\end{equation}
where $c',C_0,c_0$ are positive constants (depending on the subgaussian moment of $\xi$) and where the entries of $\Bc$ are iid copies of $\xi$.  
\end{theorem}


A similar statement for random symmetric Wigner matrices with subexponential bounds, can be found in \cite[Theorem 6.1]{Ng-dist} and \cite[Section 8]{RV-deloc}.

\begin{theorem}\label{theorem:dist:sym}

Let $H$ be a random subspace in $\R^n$ spanned by any $n-k$ column vectors of $X$ from Theorem \ref{theorem:main:sym}, with $0<k< \gamma_0 n$ for some constant $\gamma_0$ depending on the subgaussian moment of $\xi$.  There exists a deterministic subset $\CH$ of the Grassmanian of $\R^n$ where 
$$\P(H^\perp \in \CH) \ge 1- \exp(-n^{c'})$$
and such that for any subspace $H$ with $H^\perp \in \CH$ we have 
 \begin{equation}\label{eqn:dist:sym}
\P_{\Bc}(\dist(\Bc, H) \le \delta) \le (\frac{C_0 \delta}{\sqrt{k}})^k + \exp(-n^{c_0}),
\end{equation}
where $c',c_0, C_0$ are positive constants (depending on the subgaussian moment of $\xi$) and where the entries of $\Bc$ are iid copies of $\xi$.  
\end{theorem}
We refer the reader to \cite{RV-rec, Ng-dist,RV-deloc} for detailed constructions of $\CH$ basing on certain non-Diophantine properties of random structures. To continue the discussion on distances, for perturbed models considered in Theorem \ref{theorem:main:perturb:iid} and Theorem \ref{theorem:main:perturb:sym} we will extract the randomness on the column vector $\Bc$ only by using the following result from \cite[Theorem 1, Corollary 1.3]{RV-cont}.

\begin{theorem} \label{theorem:dist:cont} Let $P_{H^\perp}$ be a (deterministic) orthogonal projection in $\R^n$ onto a subspace of dimension $k\ge 1$. Let $\Bc=(x_1,\dots, x_n)$ be a random vector where $x_i$ are iid copies of a random vector $\xi$ of mean zero, variance one. Then there exist an absolute constant $C_0$ such that
$$\sup_{\Bu\in \R^n}\P\Big(\|P_{H^\perp}\Bc -\Bu\|_2 \le t \sqrt{k}\Big) \le (C_0p)^k,$$
where $p=p_{t}$ is a parameter such that $\sup_x \P(|\xi-x| \le t)\le p$. 

In particular, if $\xi$ has density function bounded by $K$ then 
$$\sup_{\Bu\in \R^n}\P\Big(\|P_{H^\perp}\Bc -\Bu\|_2 \le t \sqrt{k}\Big) \le (C_0Kt)^k.$$
\end{theorem}

Notice that the above result does not include arbitrary discrete distribution. In this case, we will replace it by the following result (see also \cite{TVcir} by Tao and Vu).

\begin{theorem} \label{theorem:dist:discrete} Let $P_{H^\perp}$ as in Theorem \ref{theorem:dist:cont}. Let $\Bc=(x_1,\dots, x_n)$ be a random vector
  where $x_i$ are iid copies of a random vector $\xi$ of mean zero, variance one and $|\xi|\le K$ with probability one. Then there exist constants $C,c_1,c_2$ depending on $\xi$ such that for any $0<t<1/2$
\begin{equation}\label{eqn:dist:CK}
\sup_{\Bu\in \R^n}\P\Big(\|P_{H^\perp}\Bc -\Bu\|_2 \le t \sqrt{k} -c_1K\Big) \le C \exp(-c_2 t^2 k).
\end{equation}
\end{theorem}

A proof of Theorem \ref{theorem:dist:discrete} is given in Appendix \ref{section:dist:discrete}.

The aforementioned ingredients are sufficient for our results with bounds implicitly depending on $k$. To work out the explicit dependences (especially for Theorem \ref{theorem:main} and Theorem \ref{theorem:main:sym}) we will be relying on several additional tools to be introduced below.

The first such ingredient is a non-gap delocalization result by Rudelson and Vershynin \cite[Theorem 1.5]{RV-deloc} for both (unperturbed) iid and  random Wigner matrices \footnote{Note that although \cite[Theorem 1.5]{RV-deloc} was stated for eigenvectors, it also extends to approximate eigenvectors as in Theorem \ref{theorem:non-gap}. See \cite[Section 4]{RV-deloc} for the reduction or \cite[Remark 2.1.8]{R} for a short discussion on this.}. For a vector $\Bx=(x_1,\dots,x_n)$, and for an index set $I\subset [n]$, we denote $\Bx_I$ by the projected vector $(x_{i_1},\dots, x_{i_{|I|}}), i_1<\dots <i_{|I|} \in I$.

\begin{theorem}\label{theorem:non-gap} Let $\gamma>0$ be given. Then there exist positive constants $\gamma_1, \gamma_2,\gamma_3$ depending on $\gamma$ such that with probability at least $1- \exp(-\gamma_1n)$ the following holds.
\begin{enumerate}
\item Let $M$ be a random iid matrix as in Theorem \ref{theorem:main}. Then or any $\Bx \in S^{n-1}$ such that $\|M\Bx\|_2 \le \gamma_2 \sqrt{n}$, for any $I \subset [n]$ with $|I| \ge \gamma n$ we have
$$\|\Bx_{I}\|_2 \ge \gamma_3.$$
\vskip .1in
\item Let $X$ be a random Wigner matrix as in Theorem \ref{theorem:main:sym}. Then for any $\Bx \in S^{n-1}$ such that $\|(X-z)\Bx\|_2 \le \gamma_2 \sqrt{n}$, for any $I \subset [n]$ with $|I| \ge \gamma n$ we have
$$\|\Bx_{I}\|_2 \ge \gamma_3.$$
\end{enumerate}
\end{theorem}

We will also work with non-random matrices. Roughly speaking we would like to obtain many well-conditioned almost square minors from a well-conditioned rectangular matrix. For this type of restricted invertibility phenomenon, we will take into account two powerful such results. The first ingredient is the main result from \cite{RV-submatrix} by Rudelson and Vershynin.

\begin{theorem}\label{theorem:submatrix}
Let $X$ be an $k \times n$ matrix with $r= \|X\|_{HS}^2/\|X\|_2^2$. Let $\eps, \delta \in (0,1)$ and let $d\le n$ be an integer such that 
$$d \ge C \frac{r}{ \eps^4 \delta} \log \frac{r}{\eps^4 \delta}.$$
Consider the $k \times d$ matrix $\tilde{X}$ which consists of $d$ unit-normalized columns of $X$ picked independently with replacement, with probabilities propositional to the squares of their Euclidean lengths. Then with probability at least $1-2 \exp(-c/\delta)$ the following holds
$$\|X X^T - \frac{\|X\|_{HS}^2}{d}\tilde{X} \tilde{X}^T \|_2 \le  \eps \|X\|_2.$$ 
\end{theorem}

The second ingredient is a more recent paper \cite[Theorem 6]{NY} by Naor and Youssef.

\begin{theorem}\label{theorem:submatrix'}
Suppose that $X$ is a full rank matrix of size $k\times d$ with $k\le d$. Then for $1\le l\le k-1$, there exists $l$ different indices $i_1,\dots, i_l$ such that the matrix $X_{(i_1,\dots, i_l)}$ with columns $\col_{i_1}(X),\dots, \col_{i_l}(X)$ has the smallest non-zero singular value $\sigma_l$ satisfying
$$\sigma_l^{-1} \le K_0\min_{r\in \{l+1,\dots, k\}} \sqrt{\frac{dr}{(r-l) \sum_{i=r}^k\sigma_i(X)^2}},$$
where $K_0$ is an absolute constant.
\end{theorem}


\section{Random iid matrices: proof of Theorem \ref{theorem:main}}\label{section:main}

\subsection{Proof of \eqref{eqn:k^2} of Theorem \ref{theorem:main}}\label{section:k^2}
 
In this section we will condition on the event $\CE_{bound}$ of \eqref{eqn:norm} and the event $\CE_{incomp}$ of Lemma \ref{lemma:comp} that for all unit vectors $\Bx$ such that $\|M\Bx\|_2 \le c \sqrt{n}$, we must have $\Bx \in \Incomp(c_0,c_1)$. 

By the min-max principle, 
$$\sigma_{n-k+1}(M) = \min_{\dim(H)=k} \max_{\Bx \in H, \|\Bx\|_2=1} \|M \Bx\|_2.$$
Thus if $\sigma_{n-k+1}(M) \le \eps/\sqrt{n}$ then there are $k$ orthogonal unit vectors $\Bz_1,\dots, \Bz_k$ such that 
\begin{equation}\label{eqn:x_i}
\|M \Bz_i\|_2 \le \eps/\sqrt{n}, 1\le i\le k.
\end{equation}
Assume that $\Bz_i =(z_{i1},\dots, z_{in})^T$, and let $\Bc_1,\dots, \Bc_n$ be the column vectors of $M$. The condition \eqref{eqn:x_i} can be read as
$$\|\sum_{j=1}^n z_{ij} \Bc_j\|_2 \le \eps/\sqrt{n}, \mbox{ for all } 1\le i\le k.$$

For short, let $Z$ be the $k \times n$ matrix spanned by the row vectors $\Bz_1^T,\dots, \Bz_k^T$, and let $\By_1,\dots, \By_n \in \R^k$ be the columns of $Z$.  For any (multi)-set $J\subset [n]$, $Z_{J}$ denotes the matrix generated by $\By_j, j\in J$.
By definition,
$$\sum_{1\le i\le n} \|\By_i\|^2_2 = k.$$
We gather here a simple fact that will be useful.
 
\begin{lemma}\label{lemma:Z:incomp} The set of unit vectors in the subspace $H_{\Bz_1,\dots, \Bz_k}$ spanned by $\Bz_1,\dots, \Bz_k$ belongs to $\Incomp(c_0,c_1)$. Consequently, for any set $J\subset [n]$ of size at least $(1-c_0)n$, the least singular value of $Z_J$ is at least $c_1$.
\end{lemma}
\begin{proof}(of Lemma \ref{lemma:Z:incomp}) Note that as $\Bz_1,\dots, \Bz_k$ are orthogonal, for any unit vector $\Bz = \sum_i \alpha_i \Bz_i \in H_{\Bz_1,\dots, \Bz_k}$, with $\mathbf{\alpha} = (\alpha_1,\dots, \alpha_k) \in S^k$, by the triangle inequality we have 
\begin{equation}\label{newbound}
\|M \Bz\|_2 \le \sum_i |\alpha_i| \|M\Bz_i\|_2 \le \eps \sqrt{k}/\sqrt{n}.
\end{equation}
As this bound is of order $o(\sqrt{n})$, and as we are under $\CE_{incomp}(c_0,c_1)$, this implies that $\Bz \in \CE_{incomp}(c_0,c_1)$, and by definition $\|\Bz_J\|_2 \ge c_1$ for any set $J\subset [n]$ of size at least $(1-c_0)n$. In other words, for any unit vector $\mathbf{\alpha}$ we have $\|\mathbf{\alpha} Z_J\|_2 \ge c_1$, as desired.
\end{proof}
Basing on this elementary fact, we show that $Z$ contains many well-conditioned minors.
\begin{lemma}\label{lemma:invert'} We can choose $\Theta_k(n^k)$ tuples $(j_1,\dots, j_k)\in [n]^k$ such that 
\begin{equation}\label{eqn:invert'}
\|Z_{(j_1,\dots, j_k)}^{-1}\|_2 =O_k(\sqrt{n}).
\end{equation}
\end{lemma}

\begin{proof}(of Lemma \ref{lemma:invert'}) We say an index $i$ is {\it bad} if  $\|\By_i\|_2 \ge c_0^{-1/2} \sqrt{\frac{k}{n}}$. As $\sum_i \|\By_i\|^2=k$, the set $W^\sharp$ of bad at most $c_0n$. Set $W$ to be the complement of $W^\sharp$. Then by Lemma \ref{lemma:Z:incomp} the matrix $Z_{W}$ has the least singular value at least $c_1$. So trivially, $\det(Z_{W}Z_{W}^T) \ge c_1^{2k}$. By the Cauchy-Binet expansion identity, 
$$\sum_{j_1,\dots, j_k \in W} |\det(X_{j_1,\dots, j_k})|^2 \ge c_1^{2k}.$$
Note also that as $\|\By_i\|_2 \le c_0^{-1/2} \sqrt{\frac{k}{n}} $. We trivially have  $|\det(Z_{(j_1,\dots, j_k)})|^2 \le \prod_{i=1}^k \|\By_{j_i}\|^2_2 \le (c_0^{-1} k/n)^k$. So there are $\Theta_k(n^k)$ tuples $(j_1,\dots, j_k)$ tuples such that $|\det(Z_{(j_1,\dots, j_k)})|^2 = \Theta(n^{-k})$. The obtained matrices $Z_{(j_1,\dots, j_k)}$ clearly satisfy the conclusion of Lemma \ref{lemma:invert'} because $|\det(Z_{(j_1,\dots, j_{k-1}), l})|^2 \le \prod_{i=1}^{k-1} \|\By_{j_i}\|^2_2 \le (c_0^{-1} k/n)^{k-1}$, where $Z_{(j_1,\dots, j_{k-1}), l}$ is the $(k-1)\times (k-1)$ matrix obtained from $Z_{(j_1,\dots, j_{k-1})}$ by eliminating its $l$-th row.
\end{proof}

Now assume that $Z_{(j_1,\dots, j_k)}$ is as in Lemma \ref{lemma:invert'}. Let $A$ be the $k\times k$ matrix $Z_{(j_1,\dots, j_k)}^T$. Recall that the event $\CE$ in \eqref{eqn:x_i} can be written as 
$$B= (\Bc_{j_1}(M),\dots, \Bc_{j_k}(M))A + (\Bc_{j_{k+1}}(M),\dots, \Bc_{j_n}(M)) A',$$
where $A'=Z_{(j_{k+1},\dots, j_n)}^T$, and where by definition each of the $k$ column vectors of $B$ has norm at most $\eps/\sqrt{n}$. 
Let $A^{-1}$ be the inverse of $A$, then by \eqref{eqn:invert'}, $\|A^{-1}\|_2 = O(\sqrt{n})$. We have
\begin{equation}\label{eqn:separation}
(\Bc_{j_1},\dots, \Bc_{j_k})+ (\Bc_{j_{k+1}},\dots, \Bc_{j_n}) A' A^{-1}= B A^{-1}.
\end{equation}
Notice that 
$$\|BA^{-1}\|_{HS} \le \|A^{-1}\|_2 \|B\|_{HS} = O(\eps).$$
Let $H$ be the subspace generated by $\col_{j_{k+1}},\dots, \col_{j_n}$. Project each of the identity from \eqref{eqn:separation} onto the orthogonal complement of $H$, after taking the norm square, we obtain
$$\sum_{i=1}^k \dist(\col_{j_i},H)^2 \le \|BA^{-1}\|_{HS}^2.$$
It then follows that
$$\dist(\col_{j_1},H)=O_k(\eps) \wedge  \dots \wedge \dist(\col_{j_k},H)=O_k(\eps).$$
By Theorem \ref{theorem:dist} (first conditioning on $H^\perp \in \CH$, and then unfolding), this event $\CE_{(j_1,\dots, j_k)}$ is bounded by 
$$\P(\CE_{(j_1,\dots, j_k)}) = \Big(O_k(\eps^k) + \exp(-\Theta(n))\Big)^k +  \exp(-\Theta(n) \le C_k \eps^{k^2} +  \exp(-\Theta(n)).$$

To complete \eqref{eqn:k^2} of Theorem \ref{theorem:main}, it remains to combine Lemma \ref{lemma:invert'} with the following elementary claim.

\begin{claim}\label{claim:averaging} Let $\{\CE_{(j_1,\dots, j_k)}, (j_1,\dots, j_k) \in [n^k]\}$ be a collection of $n^k$ events with $\P(\CE_{(j_1,\dots, j_k)}) \le p $ for all $(j_1,\dots, j_k)$. Assume that if $\CE$ holds then some $c n^k$ events $\CE_{(j_1,\dots, j_k)}$ of the collection also hold. Then
$$\P(\CE) \le c^{-1} p.$$
\end{claim}

\begin{proof}(of Claim \ref{claim:averaging}) Let $I_{(j_1,\dots, j_k)}$ be the indicator of $\CE_{(j_1,\dots, j_k)}$. Then by definition
$$\sum_{(j_1,\dots, j_k)} I_{(j_1,\dots, j_k)} \ge c n^k I_{\CE}.$$
Taking expectation, 
$$\P(\CE) \le c^{-1} n^{-k}\sum_{j_1,\dots, j_k} \P(\CE_{(j_1,\dots, j_k)}) \le c^{-1} p.$$
\end{proof}

\subsection{Proof of \eqref{eqn:assym} of Theorem \ref{theorem:main}}\label{section:l^2}
While \eqref{eqn:k^2} gives the optimal rate in terms of $\eps$, the dependence on $k$ is quite poor. Here we will try to improve on Lemma \ref{lemma:invert'} by relying on  Theorem \ref{theorem:submatrix} and Theorem \ref{theorem:submatrix'}. We will actually prove the following result.

\begin{theorem}\label{theorem:main'} For any fixed $0<\gamma<1$ there exist positive constants $C,c$ and $\gamma_0$ depending only on the subgaussian moment of $\xi$ and on $\gamma$ such that  for $1\le l\le k-1<\gamma_0 n$ and for any $\eps>0$
 \begin{equation}\label{eqn:l^2}
 \P(\sigma_{n-k+1} \le \frac{\eps}{\sqrt{n}}) \le \Big(\frac{C\eps k}{l^{1-\gamma}(k-l)}\Big)^{l^2} + \exp(-cn).
 \end{equation}
\end{theorem}

It is clear that \eqref{eqn:assym} follows from this result by taking $l$ close to $k$. To prove Theorem \ref{theorem:main'}, it suffices to prove for $k$ of order $o(n /\log n)$ because the first error term would be significantly smaller than $\exp(-cn)$ if $k$ becomes larger than this. The extra ingredient to be used  is the non-gap delocalization result from Theorem \ref{theorem:non-gap} for random iid matrix $M$. 

Together with the events $\CE_{bound}$ and $\CE_{incomp}$ in the previous section, we will also condition on the delocalization event $\CE_{deloc}$ of Theorem \ref{theorem:non-gap} for appropriate choice of $\gamma_0$. 

With $Z$ as in Section \ref{section:k^2}, we will extract from it many almost square well-conditioned minors. 

\begin{lemma}\label{lemma:d} There exist constants $\delta,c, C>0$ such that the following holds for $d=\lfloor Ck \log k \rfloor \le n$. There exists $c^d n^d$ ordered tuples $(j_1,\dots, j_d) \in [n]^d$ such that the matrix $Z_{(j_1,\dots, j_d)}$ satisfies
\begin{equation}\label{eqn:invert}
\sigma_k^{-1}(Z_{(j_1,\dots, j_d)}) \le \delta^{-1}  \sqrt{\frac{n}{d}} .
\end{equation}
\end{lemma}
\begin{proof}(of Lemma \ref{lemma:d}) Similarly to the previous section, we call in index $i$ {\it bad} if either $\|\By_i\|_2 \ge c_0^{-1/2} \sqrt{\frac{k}{n}}$ or $\|\By_i\|_2 \le c_1^2  \sqrt{\frac{k}{n}}$. As $\sum_i \|\By_i\|^2=k$, the set $W^{\sharp}$ of indices of the first type is at most $c_0n$. So by the incompressibility property,
$\sum_{i\notin W^\sharp} \|\By_i\|^2 = \sum_{j=1}^k \sum_{i\notin W^\sharp}\Bx_{ji}^2 \ge c_1^2k$. Furthermore, for the set $W^\flat$ of indices of second-type
$$\sum_{i\in W^{\flat}} \|\By_i\|^2 \le \sum_{i\in W^{\flat}} c_1^2 \frac{k}{n} \le c_1^4 k.$$
Thus, with $W= [n]\backslash (W^{\sharp}\cup W^{\flat})$, 
$$\sum_{i\in W} \|\By_i\|_2^2 \ge (c_1^2-c_1^4)k \ge c_1^2k/2.$$
So,  
$$|W| \ge c_0 c_1^2 n/2.$$
Consider the matrix $Z'$ of size $k \times |W|$  generated by the columns $\By_i, i\in W$. Recall that $c_1^2  \sqrt{\frac{k}{n}} \le \|\By_i\|_2 \le c_0^{-1/2} \sqrt{\frac{k}{n}}$. Theorem \ref{theorem:non-gap} applied to $\gamma=c_0c_1^2/2$ (where we note that the condition $\|M\Bz\|_2 \le \gamma_0 \sqrt{n}$ is fulfilled by \eqref{newbound} if $n$ is sufficiently large) implies that
$$\sigma_k(Z') \ge \gamma_3.$$
Now if $C$ is sufficiently large, Theorem \ref{theorem:submatrix} applied to $Z'$ with $\eps=\gamma_3^2/2$ yields $c^dn^d$ matrices $Z_{(j_1,\dots, j_d)}$ such that its normalized matrix $\tilde{Z}$ satisfies $\|Z' Z'^T - \frac{\|Z'\|_{HS}^2}{d}\tilde{Z} \tilde{Z}^T \|_2 \le  \gamma_3^2 \|Z'\|_2/2 \le \gamma_3^2/2$,  and so the least singular values are also comparable,
$$ \frac{\|Z'\|_{HS}^2}{d} \sigma_k(\tilde{Z}  \tilde{Z}^T) \ge \sigma_k(Z' Z'^T ) - \gamma_3^2/2 \ge \gamma_3^2-\gamma_3^2/2 = \gamma_3^2/2.$$
Note that by the process of Theorem \ref{theorem:submatrix}, the $j_1,\dots, j_d$ are not necessary distinct, but the set $\{j_1,\dots, j_d\}$ has cardinality at least $k$ in any case. Rescaling back to $Z_{(j_1,\dots,j_d)}$, we obtain
$$\sigma_k(Z_{(j_1,\dots,j_d)}) \ge c_1^2\sqrt{\frac{k}{n}}\sigma_k(\tilde{Z})   > c_1^2 \sqrt{\frac{k}{n}}\frac{\gamma_3\sqrt{d}}{2\|Z'\|_{HS}} \ge \frac{c_1^2 \gamma_3}{2} \sqrt{\frac{d}{n}}.$$
\end{proof}

Consider a matrix $Z_{(j_1,\dots,j_d)}$ that satisfies the conclusion of Lemma \ref{lemma:d}. Let $1\le l\le k-1$. Theorem \ref{theorem:submatrix'} applied to this matrix yields $l$ distinct indices $i_1,\dots, i_l \in \{j_1,\dots, j_d\}$ such that the matrix $Z_{(i_1,\dots, i_l)}$ has the smallest non-zero singular value $\sigma_l$ satisfying
\begin{equation}\label{eqn:l}
\sigma_l^{-1} \le K_0\min_{r\in \{l+1,\dots, k\}} \sqrt{\frac{\delta^{-2}rn}{(r-l)(k-r+1)}} \le 4 K_0\delta^{-1} \sqrt{\frac{kn}{(k-l)^2}},
\end{equation}
where we chose $r= \lceil \frac{k+l}{2}\rceil$.

As each such $l$-tuple $(i_1,\dots,i_l)$ may appear in at most $n^{d-l}$ $d$-tuples $(j_1,\dots,j_d)$ obtained from Lemma \ref{lemma:d}, we thus have

\begin{corollary}\label{cor:l} There exists a constant $c>0$ and $c^d n^{l}$ $l$-tuples $(i_1,\dots, i_l) \in [n]^l$ such that the matrix $Z_{(i_1,\dots, i_l)}$ satisfies \eqref{eqn:l}.
\end{corollary}

We next proceed as in the previous section. Let $A$ be the $l\times k$ matrix $Z_{\{i_1,\dots, i_l\}}^T$. Recall that the event $\CE$ in \eqref{eqn:x_i} can be written as 
$$B= (\Bc_{i_1}(M),\dots, \Bc_{i_l}(M))A + (\Bc_{i_{l+1}}(M),\dots, \Bc_{i_n}(M)) A',$$
where $A'=Z_{(i_{l+1},\dots, i_n)}^T$, and where by definition each of the $k$ column vectors of $B$ has norm at most $\eps/\sqrt{n}$. 
Let $\bar{A}$ be the $k\times l$ matrix so that $A\bar{A}=I_l$. Then by Corollary \ref{cor:l}, $\|\bar{A}\|_2 \le  4 K_0\delta^{-1} \sqrt{\frac{kn}{(k-l)^2}}$. We have
\begin{equation}\label{eqn:separation'}
(\Bc_{i_1},\dots, \Bc_{i_l})+ (\Bc_{i_{l+1}},\dots, \Bc_{i_n}) A' \bar{A}= B\bar{A}.
\end{equation}
Notice that 
$$\|B\bar{A}\|_{HS} \le \|\bar{A}\|_2 \|B\|_{HS} \le   4 C_0\delta^{-1} \sqrt{\frac{kn}{(k-l)^2}} \sqrt{\frac{k \eps^2}{n}} \le \frac{4K_0 \delta^{-1}k}{k-l}\eps.$$
Let $H$ be the subspace generated by $\col_{i_{l+1}},\dots, \col_{i_n}$. Similarly to the previous section, \eqref{eqn:separation'} implies the event $\CE_{(i_1,\dots,i_l)}$ that
$$\dist(\col_{i_1},H)^2+ \dots+ \dist(\col_{i_l},H)^2 \le (\frac{4 K_0\delta^{-1}k}{k-l})^2 \eps^2.$$

By Claim \ref{claim:averaging} and by Corollary \ref{cor:l} (and also because $d= O(k \log k)$, the factor $c^d$ does not affect the bounds),  it suffices to show that for any fixed tuple $(i_1,\dots,i_l)$
$$\P(\CE_{(i_1,\dots,i_l)}) \ll  \Big(\frac{C \eps k}{l^{1-\gamma}(k-l)}\Big)^{l^2/2} + \exp(-cn),$$
for some absolute constants $C$ and $c$.
 
\begin{lemma}\label{lemma:end:iid} Let $\CE_t$ be the event $\dist(\col_{i_1},H)^2+ \dots+ \dist(\col_{i_l},H)^2 \le t^2$. Then there exist constants $C$ and $c$ (which also depend on $\gamma$) such that
$$\P(\CE_t) \ll (C t/l^{1-\gamma})^{l^2} + \exp(-cn).$$
\end{lemma}
Observe that Lemma \ref{lemma:end:iid} follows from the following claim where the assumption is satisfied by Theorem \ref{theorem:dist}. 
\begin{claim}\label{claim:tensor} Let $t_0>0$ be any given number. Assume that $\xi_i$ are independent random variables such that $\P(|\xi_i|< t )\le (Ct/\sqrt{l})^l+\exp(-cn)$ for all $t\ge t_0$. Then for all $0<\gamma<1$ there exists a constant $C_\gamma$ such that  for all $t\ge t_0$
$$\P(\xi_1^2+\dots+\xi_l^2  <t^2)\ll  \Big(\frac{C_\gamma t}{l^{1-\gamma}} \Big)^{l^2} + \exp(-cn).$$
\end{claim}

\begin{proof}(of Claim \ref{claim:tensor}) As $\sum_i \xi_i^2 \le t^2$, there are at most $\lceil \gamma l \rceil $ indices 
such that $|\xi_i| \ge \gamma^{-1/2}t/\sqrt{l}$. For these events we use the simple bound $\P(|\xi_i| \le t)$, while for other events we use the bounds $\P(|\xi_i|\le \gamma^{-1/2}t/\sqrt{l})$. By taking union bound over at most $2^l$ possibilities, we obtain
\begin{align*}
\P(\xi_1^2+\dots+\xi_l^2  <t^2) & \le 2^l\Big((\frac{C\gamma^{-1/2} t }{l})^l+\exp(-cn)\Big)^{l -\lceil \gamma l \rceil  }  \Big((\frac{C t }{\sqrt{l}})^{l}+\exp(-cn)\Big)^{\lceil \gamma l \rceil }\\
& \ll  \Big(\frac{C_\gamma t}{l^{1-\gamma}} \Big)^{l^2} + \exp(-cn).
\end{align*}
\end{proof}

\begin{remark}\label{remark:iid:cont} When $\xi$ has density function bounded by $K$, then one applies Theorem \ref{theorem:dist:cont} instead of Theorem \ref{theorem:dist} to bound the events $\dist(\col_{i_j}, H)\le t$, conditioning on any realization of $H$.  As a consequence, we obtain a bound $(Kt/\sqrt{l})^l$ without additional terms, and so \eqref{eqn:k^2} and \eqref{eqn:assym} hold without $\exp(-cn)$.
\end{remark}

\section{Perturbed iid matrices: proof of Theorem \ref{theorem:main:perturb:iid}}\label{section:main:perturb:iid}
Similarly to the starting point of Subsection \ref{section:k^2}, by the min-max principle,  if $\sigma_{n-k+1}(M+F) \le \eps/\sqrt{n}$ then there are $k$ orthogonal unit vectors $\Bz_1,\dots, \Bz_k$ such that 
\begin{equation}\label{eqn:x_i:perturb}
\|(M+F) \Bz_i\|_2 \le \eps/\sqrt{n}, 1\le i\le k.
\end{equation}
Assume that $\Bz_i =(z_{i1},\dots, z_{in})^T$, and let $\Bc_1,\dots, \Bc_n$ be the column vectors of $M+F$. The condition \eqref{eqn:x_i:perturb} can be read as
$$\|\sum_{j=1}^n z_{ij} \Bc_j\|_2 \le \eps/\sqrt{n}, \mbox{ for all } 1\le i\le k.$$

As usual, let $Z$ be the $k \times n$ matrix spanned by the row vectors $\Bz_1^T,\dots, \Bz_k^T$, and let $\By_1,\dots, \By_n \in \R^k$ be the columns of $Z$.  For any subset $J\subset [n]$, $Z_{J}$ denotes the matrix generated by $\By_j, j\in J$.
By definition,
$$\sum_{1\le i\le n} \|\By_i\|^2_2 = k.$$

We will extract from $Z$ one well-conditioned minor. 

\begin{lemma}\label{lemma:l:perturb} Let $1\le l\le k-1$. There is at least one $l$-tuple $(i_1,\dots, i_l) \in [n]^l$ such that the matrix $Z_{(i_1,\dots, i_l)}$ of size $l\times k$ satisfies 
\begin{equation}\label{eqn:l:perturb}
\sigma_l^{-1} \le 4 K_0 \sqrt{\frac{kn}{(k-l)^2}},
\end{equation}
\end{lemma}
\begin{proof} Theorem \ref{theorem:submatrix'} applied to $Z$ yields $l$ distinct indices $i_1,\dots, i_l \in \{j_1,\dots, j_d\}$ such that the matrix $Z_{(i_1,\dots, i_l)}$ has the smallest non-zero singular value $\sigma_l$ satisfying
$$
\sigma_l^{-1} \le C_0\min_{r\in \{l+1,\dots, k\}} \sqrt{\frac{rn}{(r-l)(k-r+1)}} \le 4K_0  \sqrt{\frac{kn}{(k-l)^2}},
$$
where we chose $r= \lceil \frac{k+l}{2}\rceil$.
\end{proof}

We next proceed as in Subsection \ref{section:l^2}. Let $A$ be the $l\times k$ matrix $Z_{\{i_1,\dots, i_l\}}^T$. Recall that the event $\CE$ in \eqref{eqn:x_i:perturb} can be written as 
$$B= (\Bc_{i_1}(M+F),\dots, \Bc_{i_l}(M+F))A + (\Bc_{i_{l+1}}(M+F),\dots, \Bc_{i_n}(M+F)) A',$$
where $A'=Z_{(i_{l+1},\dots, i_n)}^T$, and where by definition each of the $k$ column vectors of $B$ has norm at most $\eps/\sqrt{n}$. 
Let $\bar{A}$ be the $k\times l$ matrix so that $A\bar{A}=I_l$. Then by Lemma \ref{lemma:l:perturb}, $\|\bar{A}\|_2 \le  4  \sqrt{\frac{kn}{(k-l)^2}}$. We have
\begin{equation}\label{eqn:separation':perturb}
(\Bc_{i_1},\dots, \Bc_{i_l})+ (\Bc_{i_{l+1}},\dots, \Bc_{i_n}) A' \bar{A}= B \bar{A}.
\end{equation}
Notice that 
$$\|B\bar{A}\|_{HS} \le \|\bar{A}\|_2 \|B\|_{HS} \le   4  K_0\sqrt{\frac{kn}{(k-l)^2}} \sqrt{\frac{k \eps^2}{n}} \le \frac{4 C_0k}{k-l}\eps.$$
Let $H$ be the subspace generated by $\col_{i_{l+1}}(M+F),\dots, \col_{i_n}(M+F)$. Equation \eqref{eqn:separation':perturb} implies the event $\CE_{(i_1,\dots,i_l)}$ that
\begin{equation}\label{eqn:pertub:iid:t}
\dist(\col_{i_1},H)^2+ \dots+ \dist(\col_{i_l},H)^2 \le t^2
\end{equation}
where 
$$t= \frac{4 K_0 k \eps}{k-l}.$$

{\bf Proof of \eqref{eqn:k^2:perturb:iid}.}  Choose $l=k-1$. Conditioned on $H$, the subspace $H^\perp$ has dimension at least $l$ (notice that the more $H$ becomes degenerate, the better bound we will get). 

With less focus on the implicit constants, we just simply bound $p_t =\sup_x \P(|\xi-x|\le t) \le 4K_0 k \sup_x \P(|\xi-x|\le \eps)= 4K_0 k p$. By Theorem \ref{theorem:dist:cont} we have 
$$\P(\dist(\col_{i_j},H)\le t) \le \P(\dist(\col_{i_j},H)\le t \sqrt{l}) \le  (C_0 p_t)^l \le (4K_0C_0k p)^l.
$$
Hence, conditioned on $H$, 
$$\P\Big(\dist(\col_{i_1},H)\le t) \wedge \dots \wedge  \dist(\col_{i_1},H)\le t \Big)  \le  (4K_0C_0k p)^{l^2}.$$
 
The proof of \eqref{eqn:k^2:perturb:iid} is then complete by unfolding the condition, and by taking the union bound over all possible $n^l$ choices of $(i_1,\dots,i_l)$.


{\bf Proof of \eqref{eqn:assym:perturb:iid}.} Here we use our deduction as above, but with $\eps$ replaced by $k \eps$ (as we are working with $(M+F)\Bz_i \le \eps k/\sqrt{n}$ now), and hence $t= 4K_0k^2 \eps/(k-l) $. 

Assume that $k$ is sufficiently large, we will choose  $l=\lfloor (1-\gamma/2)k \rfloor$. From \eqref{eqn:pertub:iid:t}, by averaging there are at least $l'=\lfloor (1-\gamma/2)l \rfloor$ indices $i$ such that $\P(\dist(\col_{i_j},H) \le \gamma^{-1/2}t/\sqrt{l})$. Thus, by taking union bound over at most $2^l$ possibilities, it boils down to estimate the event $\CE_{i_1,\dots, i_{l'}}$ that
\begin{equation}\label{eqn:small:perturb:iid}
\dist(\col_{i_1},H)=O(t/\sqrt{l}) = O(\sqrt{k} \eps)\wedge  \dots \wedge \dist(\col_{i_{l'}},H)=O(t/\sqrt{l}) = O(\sqrt{k}\eps).
\end{equation}
With $p= \sup_{x\in \R}\P(|\xi-x| \le \eps)$. Conditioned on $H$, by Theorem \ref{theorem:dist:cont}
\begin{align*}
\P(\CE_{i_1,\dots, i_{l'}}) \le (Cp)^{{l'}^2}.
\end{align*}
We then unfold the condition and take union bound over  the choices of $(i_1,\dots,i_{l'})$.

{\bf Proof of \eqref{eqn:discrete:perturb:iid}.} Assume that $k$ is sufficiently large, we proceed as in the proof of \eqref{eqn:assym:perturb:iid} above with $\eps=c_1$ with sufficiently small $c_1$,  and $t= 4K_0k^2 \eps/(k-l) $, $l=\lfloor (1-\gamma/2)k \rfloor$,  as well as $l'=\lfloor (1-\gamma/2)l \rfloor$. After obtaining \eqref{eqn:small:perturb:iid}, instead of Theorem \ref{theorem:dist:cont} we apply Theorem \ref{theorem:dist:discrete}, which yields that (conditioned on any realization of $H$), 
\begin{align*}
\P(\CE_{i_1,\dots, i_{l'}}) \le C^{l'}e^{-c_2'c_1^2 {l'}^2},
\end{align*}
completing the proof.

\section{Random symmetric matrices: proof of Theorem \ref{theorem:main:sym}}\label{section:main:sym}
\subsection{Proof of \eqref{eqn:k^2:sym}  of Theorem \ref{theorem:main:sym}}\label{section:k^2:sym}
Our starting point is similar to that of Subsection \ref{section:k^2}. We will condition on the event $\CE_{bound}$ of \eqref{eqn:norm} and the event $\CE_{incomp}$ of Lemma \ref{lemma:comp} that for all unit vectors $\Bx$ such that $\|(X-z)\Bx\|_2 \le c \sqrt{n}$, we must have $\Bx \in \Incomp(c_0,c_1)$. 

By definition, there are $k$ orthogonal unit vectors (i.e. the corresponding eigenvectors) $\Bz_1,\dots, \Bz_k$ such that 
\begin{equation}\label{eqn:x_i:sym}
\|(X-z) \Bz_i\|_2 \le \eps/\sqrt{n}, 1\le i\le k.
\end{equation}
Assume that $\Bz_i =(z_{i1},\dots, z_{in})^T$, and let $\Bc_1,\dots, \Bc_n$ be the column vectors of $M$. The condition \eqref{eqn:x_i:sym} can be read as
$$\|\sum_{j=1}^n z_{ij} \Bc_j\|_2 \le \eps/\sqrt{n}, \mbox{ for all } 1\le i\le k.$$

As we are in $\CE_{incomp}$, by Lemma \ref{lemma:Z:incomp} the set of unit vectors in the subspace $H_{\Bz_1,\dots, \Bz_k}$ spanned by $\Bz_1,\dots, \Bz_k$ belongs to $\Incomp(c_0,c_1)$.

Let $Z$ be the $k \times n$ matrix spanned by the row vectors $\Bz_1^T,\dots, \Bz_k^T$, and let $\By_1,\dots, \By_n \in \R^k$ be the columns of $Z$.  By the incompressibility property (applied for symmetric matrices), for any set $J\subset [n]$ of size at least $(1-c_0)n$, the least singular value of $Z_J$ at least $c_1$, and so Lemma \ref{lemma:invert'} applies, hence we can choose $\Theta_k(n^k)$ tuples $(j_1,\dots, j_k)\in [n]^k$ such that 
\begin{equation}\label{eqn:invert':sym}
\|Z_{(j_1,\dots, j_k)}^{-1}\|_2 =O_k(\sqrt{n}).
\end{equation}

Now assume that $Z_{(j_1,\dots, j_k)}$ is as in \eqref{eqn:invert':sym}. Let $A$ be the $k\times k$ matrix $Z_{(j_1,\dots, j_k)}^T$. Then by \eqref{eqn:invert':sym}, $\|A^{-1}\|_2 = O(\sqrt{n})$. We have
\begin{equation}\label{eqn:separation:sym}
(\Bc_{j_1},\dots, \Bc_{j_k})+ (\Bc_{j_{k+1}},\dots, \Bc_{j_n}) A' A^{-1}= B A^{-1}.
\end{equation}
Notice that 
$$\|BA^{-1}\|_{HS} \le \|A^{-1}\|_2 \|B\|_{HS} = O(\eps).$$
Let $H$ be the subspace generated by $\col_{j_{k+1}},\dots, \col_{j_n}$. Project \eqref{eqn:separation:sym} onto the orthogonal complement of $H$ and taking the squared norm, we obtain
$$\sum_{i=1}^k \dist(\col_{j_i},H)^2 \le \|BA^{-1}\|_{HS}^2.$$
It then follows that
\begin{equation}\label{eqn:small:sym}
\dist(\col_{j_1},H)=O_k(\eps) \wedge  \dots \wedge \dist(\col_{j_k},H)=O_k(\eps).
\end{equation}

Thus we have translated the event \eqref{eqn:x_i:sym} to the event of having many small distances. Now our treatment with this event $\CE_{j_1,\dots, j_k}$ will be different from the iid case as the distances are now correlated. Without loss of generality, assume that $j_i=i, 1\le i\le k$, and $H$ is the subspace generated by the last $n-k$ columns of $X$. We will remove correlations by relying on the following simple fact.

\begin{fact}\label{fact:reduction} For any $I \subset [n]$, 
$$\dist(\col_i,H) \ge \dist(\col_{i,I},H_I),$$
where $\col_{i,I}$ and $H_I$ are the projections of $\col_i$ and $H$ onto the coordinates indexed by $I$ respectively. 
\end{fact}
To exploit this fact, we observe that the event $\CE_{1,\dots, k}$ implies the event $\CF_{1,\dots, k}$ where 
\begin{align}\label{eqn:sym:F}
\CF_{1,\dots, k}:&=    \Big(\dist(\col_{1, \{2,\dots, n\}}, H_{ \{2,\dots, n\}}) = O_k(\eps) \Big)  \wedge \dots  \wedge \Big(\dist(\col_{k-2, \{k-1,\dots, n\}}, H_{ \{k-1,\dots, n\}}) = O_k(\eps)\Big) \nonumber \\ 
& \wedge  \Big(\dist(\col_{k-1, \{k,\dots, n\}}, H_{ \{k,\dots, n\}}) = O_k(\eps)\Big).
\end{align}
We next use the multiplicative rule $\P(E_{k-1}\wedge \dots \wedge E_1) =\P(E_1) \P(E_2|E_1)\dots \P(E_{k-1} | E_{k-2} \wedge \dots \wedge E_1)$. Let $\mathscr{A}_{\{i+1,\dots, n\}}$ be the the sigma-algebra generated by the $X$-entries $x_{kl}, i+1\le k,l \le n$. 

For $1\le i\le k-1$, we apply Theorem \ref{theorem:dist:sym} to obtain
$$\P\Big(\dist(\col_{i, \{i+1,\dots, n\}}, H_{ \{i+1,\dots, n\}}) = O_k(\eps) | \mathscr{A}_{\{i+1,\dots, n\}}\Big) = O_k(\eps^{k-i}) + \exp(-n^{c'}).$$
Hence,
\begin{equation}\label{eqn:EF:sym}
\P(\CE_{1,\dots, k}) \le \P(\CF_{1,\dots, k}) \le O_k(\eps^{(k-1)+\dots+1} + \exp(-n^{c'})).
\end{equation}

To complete the proof of \eqref{eqn:k^2:sym} of Theorem \ref{theorem:main:sym} we then just need to use Claim \ref{claim:averaging} again, taking into account all tuples $(j_1,\dots,j_k)$ obtained from \eqref{eqn:invert':sym}.

\begin{remark} It is the process of passing from $\CE_{j_1,\dots, j_k}$ to $\CF_{j_1,\dots,j_k}$ via Fact \ref{fact:reduction} that we lost a factor of $\eps$ in each iterative conditional substep leading to \eqref{eqn:EF:sym}. To recover the cumulative loss of $\eps^k$ in the final bound one must not use Fact \ref{fact:reduction} but work directly with $\CE_{j_1,\dots,j_k}$. This requires to work with anti-concentration of quadratic forms \cite{Ng-cont, Vershynin}. A more plausible goal is to improve the RHS of Theorem \ref{theorem:V} to $O(\eps + \exp(-n^c))$.
\end{remark}

\subsection{Proof of \eqref{eqn:assym:sym} of Theorem \ref{theorem:main:sym}}\label{section:l^2:sym}
Similar to the proof of \eqref{eqn:assym} of  Theorem \ref{theorem:main} in Subsection \ref{section:l^2}, we will condition on the event $\CE_{bound}$ of \eqref{eqn:norm}, the event $\CE_{incomp}$ of Lemma \ref{lemma:comp}, and also on the event $\CE_{deloc}$ for appropriate choice of $\gamma_0$ from the non-gap delocalization result of Theorem \ref{theorem:non-gap} applied to $X$.

Next, assume as in Section \ref{section:k^2:sym} that there are $k$ orthogonal unit vectors (eigenvectors) $\Bz_1,\dots, \Bz_k$ such that 
\begin{equation}\label{eqn:x_i:sym'}
\|(X-z) \Bz_i\|_2 \le \eps/\sqrt{n}, 1\le i\le k.
\end{equation}

Let $Z$ be the $k \times n$ matrix spanned by the row vectors $\Bz_1^T,\dots, \Bz_k^T$. With the input of Theorem \ref{theorem:non-gap}, we obtain the following analog of Corollary \ref{cor:l}

\begin{corollary}\label{cor:l:sym} With $d=\lfloor C k\log k \rfloor$ and $l=(1-\gamma/2)k$, there are $2^{-d}c^d n^{l}$ $l$-tuples $(i_1,\dots, i_l) \in [n]^d$ such that the matrix $Z_{(i_1,\dots, i_l)}$ satisfies \eqref{eqn:l}.
\end{corollary}

We proceed as in Section \ref{section:l^2}. Let $A$ be the $l\times k$ matrix $Z_{\{i_1,\dots, i_l\}}^T$. Recall that the event $\CE$ in \eqref{eqn:x_i:sym'} can be written as 
$$B= (\Bc_{i_1}(X),\dots, \Bc_{i_l}(X))A + (\Bc_{i_{l+1}}(X),\dots, \Bc_{i_n}(X)) A',$$
where $A'=Z_{(i_{l+1},\dots, i_n)}^T$, and where by definition each of the $k$ column vectors of $B$ has norm at most $\eps/\sqrt{n}$. 
Let $\bar{A}$ be the $k\times l$ matrix so that $A\bar{A}=I_l$. Then by Corollary \ref{cor:l:sym}, $\|\bar{A}\|_2 \le  4 K_0\delta^{-1} \sqrt{\frac{kn}{(k-l)^2}}$. We have
\begin{equation}\label{eqn:separation':sym}
(\Bc_{i_1},\dots, \Bc_{i_l})+ (\Bc_{i_{l+1}},\dots, \Bc_{i_n}) A' \bar{A}= B \bar{A}.
\end{equation}
As such, with $H$ being the subspace generated by $\col_{i_{l+1}},\dots, \col_{i_n}$, we are in the event $\CE_{(i_1,\dots,i_l)}$ that
$$\dist(\col_{i_1},H)^2+ \dots+ \dist(\col_{i_l},H)^2 \le (\frac{4 K_0\delta^{-1}k}{k-l})^2 \eps^2 \le (8 K_0\delta^{-1}\gamma^{-1} \eps)^2.$$

By Claim \ref{claim:averaging} and by Corollary \ref{cor:l} (as bounds of type $(c/2)^d$ are absorbed by other factors), to prove \eqref{eqn:assym:sym} of Theorem \ref{theorem:main:sym} it suffices to show that there exist absolute constants $C$ and $c$ such that
\begin{equation}\label{eqn:E:sym}
\P(\CE_{(1,\dots,l)}) \le  (\frac{C\eps}{k})^{(1-\gamma)k^2/2} +  \exp(-n^c).
\end{equation}

To this end, let $\CE_t$ be the event $\dist(\col_{1},H)^2+ \dots+ \dist(\col_{l},H)^2 \le t^2$, where $t=8 \delta^{-1}\gamma^{-1} \eps$. By averaging, there are at most $\gamma l/2$ indices 
such that $\dist(\col_{i},H)^2 \ge \gamma^{-1/2}t/\sqrt{l}$. So there are at least $l'=(1-\gamma/2)l$ indices $i$ such that $\P(\dist(\col_{i},H)^2 \le \gamma^{-1/2}t/\sqrt{l})$. Again, by taking union bound over at most $2^l$ possibilities, it boils down to estimate the event $\CE_{i_1,\dots, i_{l'}}$ that
\begin{equation}\label{eqn:small:sym'}
\dist(\col_{i_1},H)=O(t/\sqrt{l}) \wedge  \dots \wedge \dist(\col_{i_{l'}},H)=O(t/\sqrt{l}).
\end{equation}
Finally, by the argument leading to \eqref{eqn:EF:sym} basing on Theorem \ref{theorem:dist:sym} 
\begin{align*}
\P(\CE_{1,\dots, k}) \le \prod_{i=1}^{l'-1}  (\frac{C t}{\sqrt{l i}} )^{i} + \exp(-n^c) &= (\frac{C' t}{\sqrt{l l'}} )^{l'(l'-1)/2} + \exp(-n^c)\\
& \le (\frac{C\eps}{k})^{(1-\gamma)k^2/2} +  \exp(-n^c),
\end{align*}

completing the proof of \eqref{eqn:E:sym}.

\begin{remark}\label{remark:sym:cont}
When $\xi$ has density function bounded by $K$, then one applies Theorem \ref{theorem:dist:cont} instead of Theorem \ref{theorem:dist:sym} to bound the events $\dist(\col_{i_j}, H)\le t$ in Subsection \ref{section:k^2:sym} and Subsection \ref{section:l^2:sym} (conditioned on any realization of $H$).  As a consequence, we obtain bounds of type $(Kt/\sqrt{l})^l$ without additional terms. Consequently, \eqref{eqn:k^2:sym} and \eqref{eqn:assym:sym} hold without the additive terms $\exp(-n^c)$ in this case.
\end{remark}

\section{Perturbed symmetric matrices: proof of Theorem \ref{theorem:main:perturb:sym}}\label{section:main:perturb:sym}

Our approach here is similar to that of Section \ref{section:main:perturb:iid} and Section \ref{section:main:sym}, so we will be brief. Let $z$ be the midpoint of $I$. Assume that $\lambda_i \in I$, and $ (X+F)\Bz_i = \lambda_i \Bz_i$ with orthogonal eigenvectors $\Bx_i$ of norm one. Then $(X+F-z)\Bz_i = (\lambda_i-z)\Bz_i$, and so, with $X':=X+F-z$
\begin{equation}\label{eqn:x_i:perturb:sym}
\|X' \Bz_i\|_2 \le \eps/\sqrt{n}, 1\le i\le k.
\end{equation}
Assume that $\Bz_i =(z_{i1},\dots, z_{in})^T$, and let $\Bc_1,\dots, \Bc_n$ be the column vectors of $X'-z$. Let $Z$ be the $k \times n$ matrix spanned by the row vectors $\Bz_1^T,\dots, \Bz_k^T$, then by \eqref{eqn:x_i:perturb:sym} we have the following analog of Lemma \ref{lemma:l:perturb}.

\begin{lemma}\label{lemma:l:perturb:sym} Let $1\le l\le k-1$. There is at least one $l$-tuple $(i_1,\dots, i_l) \in [n]^l$ such that the matrix $Z_{(i_1,\dots, i_l)}$ of size $l\times k$ satisfies 
\begin{equation}\label{eqn:l:perturb:sym}
\sigma_l^{-1} \le 4 K_0 \sqrt{\frac{kn}{(k-l)^2}},
\end{equation}
\end{lemma}
Next, as in Subsection \ref{section:l^2:sym}, let $A$ be the $l\times k$ matrix $Z_{\{i_1,\dots, i_l\}}^T$. The event $\CE$ in \eqref{eqn:x_i:perturb:sym} can be written as  $B= (\Bc_{i_1}(X'),\dots, \Bc_{i_l}(X'))A + (\Bc_{i_{l+1}}(X'),\dots, \Bc_{i_n}(X')) A'$, where $A'=Z_{(i_{l+1},\dots, i_n)}^T$, and where each of the $k$ column vectors of $B$ has norm at most $\eps/\sqrt{n}$. 
Let $\bar{A}$ be the $k\times l$ matrix so that $A\bar{A}=I_l$. Then by Lemma \ref{lemma:l:perturb:sym}, $\|\bar{A}\|_2 \le  4 K_0 \sqrt{\frac{kn}{(k-l)^2}}$ and we also have
\begin{equation}\label{eqn:separation':perturb:sym}
(\Bc_{i_1},\dots, \Bc_{i_l})+ (\Bc_{i_{l+1}},\dots, \Bc_{i_n}) A' \bar{A}= B \bar{A}.
\end{equation}
A simple calculation shows that  $\|B\bar{A}\|_{HS} \le \frac{4 K_0k}{k-l}\eps$. Thus with $H$ being the subspace generated by $\col_{i_{l+1}},\dots, \col_{i_n}$, \eqref{eqn:separation':perturb:sym} implies the event $\CE_{(i_1,\dots,i_l)}$ that
\begin{equation}\label{eqn:perturb:sym:t}\
\dist(\col_{i_1},H)^2+ \dots+ \dist(\col_{i_l},H)^2 \le t^2, \mbox{ where } t= \frac{2K_0 k \eps}{k-l}.
\end{equation}

Without loss of generality, assume $(i_1,\dots,i_l)=(1,\dots,l)$. We then pass to distances as in Subsection 5.

{\bf Proof of \eqref{eqn:k^2:perturb:sym}.}  Choose $l=k-1$. In analogy with \eqref{eqn:sym:F} we pass from $\CE_{1,\dots,l}$ to $\CF_{1,\dots, l-1}$  
\begin{align*}\label{eqn:sym:F:perturb}
\CF_{1,\dots, l-1}:&=    \Big(\dist(\col_{1, \{2,\dots, n\}}, H_{ \{2,\dots, n\}}) \le t \Big)  \wedge \dots  \wedge  \Big(\dist(\col_{l-1, \{l,\dots, n\}}, H_{ \{l,\dots, n\}}) \le t \Big).
\end{align*}
Let $\mathscr{A}_{\{i+1,\dots, n\}}$ be the the sigma-algebra generated by the $X$-entries $x_{rs}, i+1\le r,s \le n$. For $1\le i\le k-1$, we apply Theorem \ref{theorem:dist:cont}, with $p_t=\sup_x \P(|\xi-x|\le t) \le 2K_0 k \sup_x \P(|\xi-x|\le \eps)= 2K_0 k p$, we obtain 
$$\P\Big(\dist(\col_{i, \{i+1,\dots, n\}}, H_{ \{i+1,\dots, n\}}) \le t | \mathscr{A}_{\{i+1,\dots, n\}}\Big) \le (2CK_0 k p)^{l-i}.$$
Hence by multiplicative rule $\P(\CF_{1,\dots, l-1}) \le(2CK_0 k p)^{(k-1)+\dots+1}$ as desired.


{\bf Proof of \eqref{eqn:assym:perturb:sym}.} Here we use the deduction as above but with $\eps$ replaced by $k \eps$ and $t= 2K_0k^2 \eps/(k-l) $. Choose  $l=\lfloor (1-\gamma/2)k \rfloor$. From \eqref{eqn:perturb:sym:t}, by averaging there are at least $l'= \lfloor (1-\gamma/2)l \rfloor$ indices $i$ such that $\P(\dist(\col_{i},H)^2 \le \gamma^{-1/2}t/\sqrt{l})$. Thus, by taking union bound over at most $2^l$ possibilities, it boils down to estimate the event $\CE_{1,\dots,l'}$ that
\begin{equation*}\label{eqn:small:perturb:sym}
\dist(\col_{1},H)=O(t/\sqrt{l}) = O(\sqrt{k} \eps)\wedge  \dots \wedge \dist(\col_{l'},H)=O(t/\sqrt{l}) = O(\sqrt{k}\eps).
\end{equation*}
where $H$ is the subspace generated by the last $n-k$ columns of $X'$. The rest is similar to the proof of \eqref{eqn:k^2:perturb:sym} where we use Fact \ref{fact:reduction} to pass to $\CF_{1,\dots, l'-1}=   (\dist(\col_{1, \{2,\dots, n\}}, H_{ \{2,\dots, n\}}) = O(\sqrt{k} \eps))  \wedge \dots  \wedge (\dist(\col_{l'-1, \{l',\dots, n\}}, H_{ \{l',\dots, n\}}) =O(\sqrt{k} \eps)).$ We then apply Theorem \ref{theorem:dist:cont} with $p=\sup_x \P(|\xi-x|\le \eps)$ to obtain $\P(\dist(\col_{i, \{i+1,\dots, n\}}, H_{ \{i+1,\dots, n\}}) \le O(\sqrt{k} \eps)  | \mathscr{A}_{\{i+1,\dots, n\}}) \le (Cp)^{k-i} \le (Cp)^{l'-i}$, which will then yield the upper bound $\P(\CF_{1,\dots, l-1}) \le(Cp)^{l'(l'-1)/2}$.

{\bf Proof of \eqref{eqn:discrete:perturb:sym}.} Assume that $k$ is sufficiently large, we proceed as in the proof of \eqref{eqn:assym:perturb:sym} with $\eps=c_1$ where $c_1$ is sufficiently small, and with  $t= 2K_0k^2 \eps/(k-l), l=\lfloor (1-\gamma/2)k \rfloor$ and $l'= \lfloor (1-\gamma/2)l \rfloor$. After using Fact \ref{fact:reduction} to pass to  $\CF_{1,\dots, l'-1}$, we apply Theorem \ref{theorem:dist:discrete} (noting that $H_{ \{i+1,\dots, n\}}$ has co-dimension at least $k-i> \gamma k/2$) to obtain $\P(\dist(\col_{i, \{i+1,\dots, n\}}, H_{ \{i+1,\dots, n\}}) \le O(\sqrt{k} c_1)  | \mathscr{A}_{\{i+1,\dots, n\}}) \le  C e^{-c_2'c_1^2 (k-i)}$.

{\bf Acknowledgement.} The author is grateful to M.~Rudelson for suggesting Theorem \ref{theorem:main}, and for his invaluable comments and other suggestions. He is also thankful to P.~Youssef, and especially to the anonymous referee whose suggestions have help substantially improve the presentation of the note.

\appendix

\section{Proof of Theorem \ref{theorem:dist:cont}}\label{section:dist:discrete}

Recall that
$$\P(|\xi| \le K) =1.$$ 
Notice that the case $\Bu=0$ follows from Talagrand's concentration bound \cite{Talagrand} as $\|P_{H^\perp} \Bc \|_2$ is convex and 1-Lipschitz (see also \cite{TVcir}). 
\begin{theorem}\label{theorem:Tal} For any $t> 0$
$$\P(|\|P_{H^\perp} \Bc \|_2 - \sqrt{k}|\ge t \sqrt{k}) \le C \exp (-c_2 t^2 k),$$
for some constants $C,c_2$ that depend on $K$.
\end{theorem}
Now we consider two cases.

{\bf Case 1.} When $\|\Bu\|_2 \le (1-2t)\sqrt{k}$ or $\|\Bu\|_2 \ge (1+2t)\sqrt{k}$. Then if  $\|P_{H^\perp}\Bc -\Bu\|_2 \le t \sqrt{k}$, by the triangle inequality we must have either $\|P_{H^\perp}\Bc\|_2 \le (1-t) \sqrt{k}$ or  $\|P_{H^\perp}\Bc\|_2 \ge (1+t) \sqrt{k}$, and so \eqref{eqn:dist:CK} follows from Theorem \ref{theorem:Tal}.

{\bf Case 2.} When $(1-2t)\sqrt{k} \le \|\Bu\|_2 \le (1+2t)\sqrt{k}$. Consider the random variable $X = |\|P_{H^\perp} \Bc -\Bu \|_2 $. This function of $\Bc$ is again convex  (as  $\|P_{H^\perp} (x\Bc_1 + (1-x) \Bc_2) -\Bu \|_2 = \|x(P_{H^\perp} \Bc_1-\Bu) + (1-x) (P_{H^\perp} \Bc_2-\Bu) \|_2 \le x \|P_{H^\perp} \Bc_1-\Bu \|_2 + (1-x)  \|(P_{H^\perp} \Bc_2-\Bu \|_2$) and $1$-Lipschitz, by Talagrand's concentration result, with $M=M(X)$ being the median of $X$
\begin{equation}\label{eqn:XM}
\P(|X-M| \ge t \sqrt{k}) \le C\exp(-c_2 t^2k),
\end{equation}
for some constants $C, c_2$ that depend on $K$.

We need to specify the median of $X$.  Notice that 
$$X^2 =  \Bc^T P_{H^\perp} \Bc - 2 \Bc^T  P_{H^\perp} \Bu + \|\Bu\|_2^2.$$
So $\E X^2 =  k + \|\Bu\|_2^2$. We write
$$Y:=X^2-  \E X^2= X^2 - ( k + \|\Bu\|_2^2 ) = \sum_{i,j} p_{ij} (x_i x_j -\delta_{ij}) - 2 (\sum_i x_i v_i),$$
where $P_{H^\perp} = (p_{ij})$ and $\Bv=(v_1,\dots, v_n) =  P_{H^\perp} \Bu$. 

After squaring out and as $x_i$ are iid and $\E \xi=0, \E \xi^2=1$ as well as $|\xi|\le K$, we have
\begin{align*}
\E Y^2&=  \E (\sum_{i,j} p_{ij} (x_i x_j -\delta_{ij}))^2  + 4 \E (\sum_i x_i v_i)^2  -4 \E \sum_i v_i p_{ii} (x_i^3-x_i)\\
&= O(\sum_{i,j} p_{ij}^2 K^2) +O(\sum_i v_i^2) + O(K \sqrt{(\sum_i v_i^2)(\sum_i p_{ii}^2)} )\\
&= O(K^2 k),
\end{align*}
where we used the fact that $\sum_i v_i^2 \le \|\Bu\|_2^2 \le (1+2t)^2k$ and $\sum_{ij} p_{ij}^2 =\sum_i p_{ii} =k$.

Thus by Markov's inequality the median of $|Y|$ is $O(K \sqrt{k})$, and so the median of $X^2$ is at least $( k + \|\Bu\|_2^2 ) - O(K\sqrt{k})$. Consequently, the median $M(X)$ of $X$ is at least $\sqrt{ k + \|\Bu\|_2^2 } - O(K)$. Substitute into \eqref{eqn:XM}, with sufficiently small $t$, we obtain
$$\P\Big(X\le  t \sqrt{k} -O(K)\Big) \le \P\Big(X\le \sqrt{ k + \|\Bu\|_2^2 } - O(K) - t \sqrt{k}\Big) \le \P\Big(|X-M| \ge t \sqrt{k}\Big) \le C\exp(-c_2 t^2k).$$

\end{document}